\documentclass[11pt,reqno]{amsart}
\usepackage[utf8]{inputenc}
\usepackage{amsmath}
\usepackage{amsfonts}
\usepackage{amssymb}
\usepackage{amsthm}
\usepackage{hyperref}
\usepackage{multicol}
\usepackage{pictexwd, dcpic}
\usepackage{fancyhdr}
\usepackage{mathrsfs}
\usepackage{latexsym,mathtools}
\usepackage{bm, enumerate}
\usepackage{graphicx}
\usepackage{color,soul}
\usepackage{xcolor}
\usepackage{hyperref}
\usepackage{dsfont}
\usepackage{todonotes}
\usepackage{nicematrix}
\usepackage{tikz-cd}
\makeatletter
\@namedef{subjclassname@2020}{%
	\textup{2020} Mathematics Subject Classification}
\makeatother

\newcommand{\beq}{\begin{eqnarray}}
	\newcommand{\eeq}{\end{eqnarray}}
\newcommand{\beqn}{\begin{eqnarray*}}
	\newcommand{\eeqn}{\end{eqnarray*}}
\newcommand{\rar}{\rightarrow}

\reversemarginpar  % These 2 commands work very fine

\numberwithin{equation}{section}

\newtheorem{theorem}{\bf Theorem}[section]
\newtheorem{proposition}[theorem]{\bf Proposition}%[section]
\newtheorem{lemma}[theorem]{\bf Lemma}%[section]

\theoremstyle{remark}
\newtheorem{definition}[theorem]{\bf Definition}%[section]
\newtheorem{example}[theorem]{\bf Example}
\newtheorem{remark}[theorem]{\bf Remark}%[section]

\newcommand{\D}{\mathbb{D}}
\newcommand{\N}{\mathbb{N}}
\newcommand{\C}{\mathbb{C}}

\def\D{\mathbb{D}}

\theoremstyle{definition}

\theoremstyle{remark}

\hypersetup{
	colorlinks,
	citecolor=blue,
	linkcolor=blue,
	urlcolor=blue}

\title{}
\begin{document}

%\title[Function theory on the unit ball of $c_0$]{Holomorphic functions on the unit ball of $c_0$: realization, Schur-Agler decomposition and interpolation}
\title[The operator-valued Schur-Agler class in infinitely many variables]{The operator-valued Schur-Agler class in infinitely many variables}

\author[S. Rastogi]{Shubham Rastogi}
\address{Statistics and Mathematics Unit, Indian Statistical Institute, 8th Mile, Mysore Road,
Bangalore, 560059, India}
\email{shubhamrastogi4ever@gmail.com, shubham\_pd@isibang.ac.in}

\author[C. K. Sahu]{Chaman Kumar Sahu}
\address{Department of Mathematics, Indian Institute of Technology Bombay,
Powai, Mumbai, 400076, India}
\email{sahuchaman9@gmail.com, chamanks@math.iitb.ac.in}

\subjclass[2020]{}
\keywords{}

\begin{abstract}
Motivated by a recent work of G. Knese [Canad. Math.  Bull., 2026], we introduce the operator-valued Schur-Agler class in infinitely many variables, consisting of functions whose restrictions to every finite-dimensional polydisk lie in the operator-valued Schur-Agler class. We first establish an operator-valued analogue of Montel's theorem and, using it as a key ingredient, establish an Agler decomposition and a realization formula for functions in this class. The series in the Agler decomposition converges in the strong operator topology, and we show that this convergence cannot, in general, be strengthened to norm convergence. The realization formula yields a holomorphic extension of every function in the operator-valued Schur-Agler class to $\mathrm{Ball}(\ell^\infty)$, the open unit ball of $\ell^\infty$. Finally, as an application, we prove an Agler-Pick interpolation theorem in infinitely many variables.
\end{abstract}

\maketitle

\begin{center}
\begin{minipage}{0.8\textwidth}
\centering

\textbf{Notations}

\vspace{2mm}

\begin{tabular}{ll}
$\mathbb C$
& Complex plane.\\[2mm]
$\mathbb D$
& Open unit disc centered at $0$ in $\mathbb C$.\\[2mm]
$\mathbb D^n$
& Open unit polydisc $\{(z_1,\ldots,z_n)\in\mathbb C^n:z_i\in \D,\ i=1,\ldots,n\}$.\\[2mm]
$\mathbb N$
& Set of natural numbers.\\[2mm]
$\mathbb Z_+$
& Set of all non-negative integers.\\[2mm]
$\mathcal E,\mathcal E_*$
& Complex separable Hilbert spaces.\\[2mm]
$\mathcal L(\mathcal E,\mathcal E_*)$
& Space of bounded linear operators from $\mathcal E$ to $\mathcal E_*$.\\[2mm]
$\mathcal L(\mathcal E)$
& $\mathcal L(\mathcal E,\mathcal E)$.\\[2mm]
$\mathcal X, \mathcal Y, \mathcal Z$
& Normed spaces.\\[2mm]
$\operatorname{Ball}(\mathcal X)$
& $\{x\in\mathcal X:\|x\|<1\}$.\\[2mm]
$\overline{\operatorname{Ball}}(\mathcal X)$
& $\{x\in\mathcal X:\|x\|\leq1\}$.\\[2mm]
$H^\infty(\operatorname{Ball}(\mathcal Z),\mathcal Y)$
& Space of bounded $\mathcal Y$-valued holomorphic functions on
  $\operatorname{Ball}(\mathcal Z)$.\\[2mm]
$\ell^2$
& $\left\{(z_j)_{j= 1}^\infty:z_j\in \C, \displaystyle\sum_{j=1}^{\infty}|z_j|^2<\infty\right\}$.\\[2mm]
$\mathbb D_2^\infty$
& The Hilbert multidisc $\left\{(z_j)_{j= 1}^\infty \in \ell^2 :|z_j|<1 \text{ for every } j
\right\}$.\\[2mm]
$\ell^\infty$
& Space of all bounded sequences of complex numbers.\\[2mm]
$c_0$
& Space of complex sequences converging to zero.\\[2mm]
$c_{00}$
& Space of eventually zero sequences of complex numbers.
\end{tabular}
\end{minipage}
\end{center}
\renewcommand{\thefootnote}{\fnsymbol{footnote}}
\footnotetext{MSC: 47A56, 46G20, 32A19, 47A13, 47A48, 47A57, 32A70, 32E30.\\ Keywords: Operator-valued Schur-Agler class, Agler decomposition, realization formula, Montel's theorem, holomorphy in infinite dimensions.}

\section{Introduction}
A classical result in function theory states that %every Schur class function, that is, 
every bounded holomorphic function $f:\D\to \mathcal L(\mathcal E, \mathcal E_*)$ with $\|f\|_\infty\leq 1$ admits a realization formula. More precisely, there exist a Hilbert space $\mathcal H$ and a unitary colligation matrix $U= \begin{pmatrix}
A & B\\
C& D
\end{pmatrix}: \mathcal E \oplus \mathcal H\rightarrow  \mathcal E_* \oplus\mathcal H$ such that
$$
f(z) = A+ zB(I-zD)^{-1}C, \quad z \in \D.
$$
Over the years, realization formulas have been obtained in 
various settings beyond the open unit disc; see, for example, \cite{AM, AY, BB, BBF, BG, BK-V, BS, DM, Knese}.

A function $f: \mathbb D^n \rightarrow \mathcal L(\mathcal E,\mathcal E_*)$ is said to be in the {\it Schur class} $\mathcal S_n(\mathcal E, \mathcal E_*)$ if $f$ is holomorphic and $\text{sup}_{z \in \mathbb D^n} \|f(z)\| \leq 1.$  When $\mathcal E=\mathcal E_*=\C,$ we write $\mathcal S_n$ for $\mathcal S_n(\mathcal E,\mathcal E_*)$. While every Schur class function on $\D$ or $\D^2$ admits a realization formula, this phenomenon fails on $\D^n$ for $n\geq 3$. To recover a realization theory, one is led to a distinguished subclass of the Schur class, namely the Schur-Agler class, whose elements do admit realization formulas \cite{Agler}. Given $n \in \mathbb N,$ a function $f\in \mathcal S_n(\mathcal E, \mathcal E_*)$ is said to be in the {\it Schur-Agler class} $\mathcal {SA}_n(\mathcal E, \mathcal E_*)$ if for every $n$-tuple $T=(T_j)_{j=1}^n$ of commuting strict contractions, $\|f(T)\| \leq 1.$ Here, $f(T)$ is defined via the tensor functional calculus as
\[
f(T):=\sum_{k\in\mathbb{Z}_+^n} A_k\otimes T^k,
\]
where
\[
f(z)=\sum_{k\in\mathbb{Z}_+^n}A_k z^k
\]
is the power series expansion of $f$. Above, for \(k=(k_1,\ldots,k_n)\in\mathbb{Z}_+^n\), we use the notation
\[
z^k=z_1^{k_1}\cdots z_n^{k_n}
\qquad\text{and}\qquad
T^k=T_1^{k_1}\cdots T_n^{k_n}.
\] 
%It is well known that $\mathcal S_n(\mathcal E,\mathcal E_*)=\mathcal {SA}_n(\mathcal E,\mathcal E_*)$ for $n\leq 2$. For $n>2$, Varopoulos \cite{Var} exhibited counterexamples showing that these two classes do not coincide.  
The Schur-Agler class has been the subject of several recent investigations; see, for example, \cite{BBD, BK, DS, K}. J. Agler \cite{Agler} established an Agler decomposition and a realization formula for functions in this class. This result is central to the developments in the present article. More precisely, he proved the following theorem; see also \cite{BK-V, BT}.

\begin{theorem}
\label{Agler decomposition realization}
Let $f:\D^n \rightarrow \mathcal L(\mathcal E, \mathcal E_*)$ be a function. The following are equivalent.

\begin{itemize}
\item [(1)] $f\in \mathcal {SA}_n(\mathcal E, \mathcal E_*)$.

\item [(2)] For each $1\leq j\leq n$, there exist an auxiliary Hilbert space $\mathcal H_j$ and a holomorphic operator-valued function $T_j: \D^n\rightarrow \mathcal L(\mathcal H_j, \mathcal E_*)$ such that $$I_{\mathcal E_*}-f(z)f(w)^*=\sum_{j=1}^{n}(1-z_j\overline{w_j})T_j(z)T_j(w)^*, \quad z, w \in \mathbb D^n.$$

\item [(3)] For each $1\leq j\leq n$, there exists an $\mathcal L(\mathcal E_*)$-valued kernel $K_j$ on  $\D^n$ such that $$I_{\mathcal E_*}-f(z)f(w)^*=\sum_{j=1}^{n}(1-z_j\overline{w_j})K_j(z, w), \quad z, w \in \mathbb D^n.$$

\item [(4)] There exist auxiliary Hilbert spaces $\mathcal K_j$ $($$1\leq j\leq n$$)$ and a contraction $U= \begin{pmatrix}
    A & B\\
    C& D
\end{pmatrix}: \mathcal E \bigoplus \oplus_{j=1}^{n}\mathcal K_j \rightarrow  \mathcal E_* \bigoplus \oplus_{j=1}^{n} \mathcal K_j $ such that
$$
f(z) = A+ B(I-\Delta(z)D)^{-1}\Delta(z)C, \quad z \in \D^n,
$$ where $\Delta(z)=\sum_{i=1}^{n}z_iP_i.$ Here, $P_i$ denotes the orthogonal projection of $\oplus_{j=1}^{n}\mathcal K_j$ onto $\mathcal K_i$ for each $1\leq i\leq n.$
\end{itemize}
\end{theorem}
The realization formula provides a powerful framework for proving results such as the Nevanlinna-Pick interpolation theorem \cite{AM}, the commutant lifting theorem \cite{BTV}, and the Carathéodory approximation theorem \cite{ABJK}. In addition, realization-theoretic methods have found applications in engineering; see \cite{G,H}.

G. Knese studied scalar-valued Schur-Agler functions in infinitely many variables (see also \cite{DM}, \cite[Section 7]{AL}). For this class of functions, he obtained an Agler decomposition and a realization formula \cite{Knese} (cf. \cite{DM}). Knese's proof uses the scalar-valued Montel theorem in infinitely many variables; see \cite[Theorem 2.17]{DGMS} for this theorem. Motivated by this work, we introduce operator-valued Schur-Agler functions in infinitely many variables below.
\begin{definition}
A function $f: \text{Ball}(c_{00}) \rightarrow \mathcal L(\mathcal E, \mathcal E_*)$ is said to be in the {\it Schur-Agler class} $\mathcal {SA}_\infty(\mathcal E, \mathcal E_*)$ if $f|_{\mathbb D^n} \in \mathcal {SA}_n(\mathcal E, \mathcal E_*)$ for every $n \in \mathbb N.$  
\end{definition}
Given $f\in \mathcal {SA}_n(\mathcal E, \mathcal E_*)$, one obtains a function $\tilde f\in\mathcal {SA}_\infty(\mathcal E, \mathcal E_*)$ by setting
$$\tilde f\big((z_k)_{k\in\mathbb N}\big):=f(z_1,\ldots,z_n).
$$
Thus, the class ${SA}_\infty(\mathcal E, \mathcal E_*)$ contains, in particular, all finite-variable Schur-Agler functions viewed as functions of infinitely many variables. However, the class $\mathcal{SA}_\infty(\mathcal E,\mathcal E_*)$ also contains functions that genuinely depend on infinitely many variables. An example of such a function is given in Example \ref{exam}.

In this article, we establish an Agler decomposition and a realization formula for functions in $\mathcal {SA}_\infty(\mathcal E, \mathcal E_*)$; see Theorem 
\ref{Auto-Ext}. The passage from the scalar-valued to the operator-valued setting raises two fundamental issues: establishing an appropriate analogue of Montel's theorem and determining a suitable mode of convergence for the Agler decomposition. We first establish an operator-valued analogue of Montel's theorem; see Theorem \ref{Montel}, where the convergence is  obtained in the weak sense. Despite this weaker mode of convergence, we show that the series appearing in the Agler decomposition converges in the strong operator topology (SOT). Moreover, Example \ref{exam} shows that this convergence cannot, in general, be strengthened to norm convergence. Thus, the SOT convergence in the Agler decomposition is sharp.%In particular, the series appearing in the Agler decomposition converges in the strong operator topology. We further show that this convergence cannot, in general, be strengthened to norm convergence, as demonstrated by Example \ref{exam}. Thus, the strong operator topology is sharp for the Agler decomposition in this setting. In the course of this development, we also establish an operator-valued version of Montel's theorem for functions on the unit ball of a separable Banach space; see Proposition \ref{Montel}.

We next turn to the interpolation problem. In the polydisc setting, the Agler-Pick interpolation theorem is well known and takes the following form; see \cite{Agler, A-M}.

\begin{theorem}
% \label{A_P}
Let $X$ be a finite subset of $\D^n$ and $f: X \to \overline{\text{Ball}}(\mathcal L(\mathcal E,\mathcal E_*))$ be a function. Then, the following are equivalent.
\begin{itemize}
\item [(i)] There exists $\tilde{f} \in \mathcal{SA}_n(\mathcal E, \mathcal E_*)$ with $\tilde{f}|_{X} = f$.
\item [(ii)] There exist  $\mathcal L(\mathcal E_*)$-valued kernels $K_j$ $($$1\leq j\leq n$$)$ on $X$ such that $$
I_{\mathcal E_*} - f(z)f(w)^* = \sum_{j=1}^{n} (1-z_j \overline{w_j}) K_j(z, w), \quad z, w \in X.
$$
\end{itemize}
\end{theorem}
We establish an Agler-Pick interpolation theorem in infinitely many variables; see Theorem~\ref{A_P_infinite}. This provides an interpolation criterion for functions in $\mathcal {SA}_\infty(\mathcal E, \mathcal E_*)$ and extends \cite[Theorem 1.7]{Knese} from the scalar-valued setting to the operator-valued setting.

The article is organized as follows. Section \ref{Pre} establishes an operator-valued version of Montel's theorem for functions on the unit ball of a separable Banach space. In Section \ref{Realization and Decomposition}, we derive an Agler decomposition and a realization formula for functions in $\mathcal{SA}_\infty(\mathcal E,\mathcal E_*)$. Section \ref{Agler-Pick} is devoted to an Agler-Pick interpolation theorem for operator-valued functions in infinitely many variables.

\section{An operator-valued version of Montel's theorem }\label{Pre}
We first collect the necessary material for our purposes. Let $U$ be an open subset of a normed space $\mathcal{X}$. A function $f:U\to \mathcal{L}(\mathcal{E}, \mathcal{E}_*)$ is said to be \emph{G\^ateaux holomorphic} on $U$ if, for every $z\in U$ and $h\in \mathcal X$, the function
\[\Omega:=\{\lambda\in\mathbb{C}:z+\lambda h\in U\}\to \mathcal{L}(\mathcal{E}, \mathcal{E}_*),\qquad
\lambda\mapsto f(z+\lambda h),\]
is  (Fr\'echet) holomorphic on $\Omega$. Given $f: U \to \mathcal{L}(\mathcal{E}, \mathcal{E}_*)$ and vectors $x \in \mathcal{E}, y \in \mathcal{E}_*$, we define the scalar-valued function $f^{x,y}: U \to \mathbb{C}$ by
\begin{equation*}
% \label{f-x-y}
    f^{x,y}(z)=\langle f(z)x, y \rangle_{\mathcal E_*}, \quad z \in U.
\end{equation*}

\begin{proposition}\label{holocheck}
Let $U$ be an open subset of a Banach space $\mathcal X$ and let $f:U \rightarrow \mathcal L(\mathcal E, \mathcal E_*)$ be a locally bounded function. Then, $f$ is holomorphic if and only if for each $x \in \mathcal E$ and $y \in \mathcal E_*$, the function $f^{x,y}$ is holomorphic.
\end{proposition}
\begin{proof}
The forward implication is immediate.
%It is easy to see that if $f$ is holomorphic, then $f^{x,y}$ is holomorphic for all $x\in \mathcal E$ and $y\in \mathcal E_*.$ 
Conversely, let $f^{x,y}$ be holomorphic for all $x\in \mathcal E$ and $y\in \mathcal E_*.$ We first show that $f$ is G\^ateaux holomorphic. Let $z\in U$ and $h\in \mathcal X.$ Define $\Omega:=\{\lambda \in \C: z+\lambda h \in U\}$. It is clear that $\Omega$ is an open set. Let $g: \Omega \to \mathcal L(\mathcal E, \mathcal E_*)$ be given by $g(\lambda)=f(z+\lambda h)$, $\lambda \in \Omega.$ Let $x\in\mathcal E, y\in \mathcal E_*$ and $\lambda_0\in \Omega.$ Since $f^{x,y}$ is holomorphic at the point $z+\lambda_0h,$ there exists a continuous linear functional $(z+\lambda_0h)^*\in \mathcal X^*$ such that 
\begin{align}\label{holofxy}
\lim_{h^{'} \to 0}\frac{\langle f(z+\lambda_0h+ h^{'}) x-f(z+\lambda_0h)x,y\rangle-(z+\lambda_0h)^*(h^{'})}{\|h^{'}\|}=0.
\end{align}
Now for small enough $\lambda \in  \C,$ we have
\begin{align*}
&\Big|\frac{\langle f(z+\lambda_0h+\lambda h)x-f(z+\lambda_0 h)x,y \rangle- \lambda(z+\lambda_0 h)^{*}(h)}{\lambda}\Big|\\
&=\|h\|\Big|\frac{\langle f(z+\lambda_0h+\lambda h)x-f(z+\lambda_0 h)x,y \rangle- (z+\lambda_0 h)^{*}( \lambda h)}{\|\lambda h\|}\Big|.
\end{align*}
By putting $\lambda h = h^{'}$ and sending $\lambda \to 0,$ we obtain by eq. (\ref{holofxy}) that $(z+\lambda_0 h)^*(h)$ is the differential of $g^{x,y}$ at $\lambda_0.$ Therefore, $g^{x,y}$ is holomorphic on $\Omega$ for all $x\in \mathcal E$ and $y\in \mathcal E_*.$ Hence, by \cite[Lemma 3.1]{AKM}, $g$ is holomorphic on $\Omega.$ This proves that $f$ is G\^ateaux holomorphic. 

Define $$A=\{h\in \mathcal L(\mathcal E,\mathcal E_*)^{*}: h(T)=\langle Tx, y\rangle \  \text{for some}\ x\in \overline{Ball}(\mathcal E), y\in \overline{Ball}(\mathcal E_*) \} .$$ It is immediate that $A$ is contained in the closed unit ball of $\mathcal L(\mathcal E,\mathcal E_*)^{*}$. The proof now follows from \cite[Lemma 15.44]{DGMS}, as $A$ is a norming set for $\mathcal L(\mathcal E,\mathcal E_*)$. 
\end{proof}

We are now ready to state the following operator-valued version of Montel's theorem. To the best of our knowledge, this result has not previously appeared in the literature. The reader is referred to \cite[Theorem 2.17]{DGMS} for the scalar version and \cite[Lemma 3.2]{AKM} for the operator-valued version when domain is a pseudoconvex domain in $\C^n.$ 

\begin{theorem}
[Montel's theorem for operator-valued case]\label{Montel}
Let $\mathcal X$ be a separable Banach space. Let $(f_n)_{n\geq 1}$ be a bounded sequence in $H^\infty(\text{Ball}(\mathcal X), \mathcal L(\mathcal E, \mathcal E_*))$. Then there exist a subsequence $(f_{n_k})_{k\geq 1}$ of $(f_n)_{n\geq 1}$ and a function $f \in H^\infty(\text{Ball}(\mathcal X), \mathcal L(\mathcal E, \mathcal E_*))$ such that for each $x \in \mathcal E$ and $y \in \mathcal E_*,$  the sequence of functions $(f_{n_k}^{x,y})_{k\geq 1}$ converges to $f^{x,y}$ uniformly on each compact subset of $\text{Ball}(\mathcal X).$
\end{theorem}
\begin{proof}
Let $A$ and $B$ be countable dense subsets of $\mathcal E$ and $\mathcal E_*$, respectively. By \cite[Theorem 2.17]{DGMS} and  a Cantor diagonal argument, there exist a subsequence $(n_k)_{k\geq 1}$ of $\N$ and functions $f(x,y)\in H^\infty(\text{Ball}(\mathcal X), \C)$ ($x\in A, y\in B$) such that for all $x\in A$ and $y\in B,$ the sequence of functions $(f_{n_k}^{x,y})_{k\geq 1}$ converges to $f(x,y)$ uniformly on each compact subset of $\text{Ball}(\mathcal X).$ Since the sequence $(f_{n_k})_{k\geq 1}$ is uniformly bounded and $A, B$ are dense subsets, we get that  the $\lim_{k\to \infty}f_{n_k}^{x,y}(z)$ exists for all $z\in \text{Ball}(\mathcal X)$ and for all $x\in \mathcal E, y\in \mathcal E_*.$ For  $z\in \text{Ball}(\mathcal X)$, define the sesquilinear form $$u_z:\mathcal E\times \mathcal E_*\to \C$$ by $$u_z(x,y)=\lim_{k\to \infty}\langle f_{n_k}(z)x,y\rangle.$$ Let $M:=\sup\{\|f_n\|_{\infty}: n\geq 1\}.$ Since for any $k\in \N,$ $|\langle f_{n_k}(z)x,y\rangle|\leq M\|x\|\|y\|,$ we have $|u_z(x,y)|\leq M\|x\|\|y\|.$ Therefore, $u_z$ is a bounded sesquilinear form. This implies that there exists a  unique bounded operator $T_z: \mathcal E\to \mathcal E_*$ with $\|T_z\|\leq M$ such that $u_z(x,y)=\langle T_z x,y\rangle$
for all $x\in \mathcal E, y\in \mathcal E_*.$ Define $f: \text{Ball}(\mathcal X)\to \mathcal L(\mathcal E,\mathcal E_*)$ by $$f(z)=T_z.$$ Clearly, $\|f(z)\|\leq M$ for $z\in \text{Ball}(\mathcal X).$ For $x\in A, y \in B,$ since $f^{x,y}=f(x,y)$, we get that $f^{x,y}$ is holomorphic. For $x\in \mathcal E, y \in \mathcal E_*,$ there exist sequences $(x_n)_{n\geq1}\subset A$ and $(y_n)_{n\geq1}\subset B$ such that $(x_n)_{n\geq1}$ converges to $x$ and $(y_n)_{n\geq1}$ converges to $y$. Using the fact that $f$ is bounded, we have that $(f^{x_n,y_n})_{n\geq 1}$ converges to $f^{x,y}$ uniformly on $\text{Ball}(\mathcal X).$ Since $(f^{x_n,y_n})_{n\geq 1}$ is a bounded sequence of holomorphic functions, we get that $f^{x,y}$ is holomorphic by \cite[Theorem 2.13]{DGMS}. Hence, by Proposition \ref{holocheck}, $f$ is holomorphic. 

Next, we show that for each $x \in \mathcal E$ and $y \in \mathcal E_*,$  the sequence of functions $(f_{n_k}^{x,y})_{k\geq 1}$ converges to $f^{x,y}$ uniformly on each compact subset of $\text{Ball}(\mathcal X).$ For $x\in A, y\in B,$ as $f^{x,y}=f(x,y),$ it follows that $(f_{n_k}^{x,y})_{k\geq 1}$ converges to $f^{x,y}$ uniformly on each compact subset of $\text{Ball}(\mathcal X).$ Using this together with the fact that $f$ is bounded and $A, B$ are dense subsets, we get the desired convergence for all $x\in \mathcal E, y\in \mathcal E_*.$ This completes the proof.
\end{proof}

Since a bounded sequence of bounded operators need not have an SOT-convergent subsequence, the convergence in Theorem \ref{Montel} cannot, in general, be strengthened to convergence in the SOT sense.

We record the following form of Hilbert's criterion, which will be useful in Section \ref{Realization and Decomposition}. It follows directly from \cite[Theorem 15.26]{DGMS}.

\begin{proposition}[Hilbert's criterion]\label{Hilbert criterion}
Let $(f_n)_{n=1}^\infty$ be a sequence of functions such that $f_n \in \mathcal S_n(\mathcal E, \mathcal E_*)$ for every $n$, and 
$f_n(z_1, z_2, \ldots, z_m, 0, 0, \ldots) = f_m(z_1, z_2, \ldots, z_m),~ n > m$. Then, there exists a unique holomorphic function $f: \text{Ball}(c_0) \rightarrow \overline {\text{Ball}}(\mathcal L(\mathcal E, \mathcal E_*))$ such that $$f(z_1, z_2, \ldots, z_n, 0, 0, \ldots) = f_n(z_1, z_2, \ldots, z_n), \quad n \geq 1.$$
\end{proposition}

\section{Realization formula and Agler decomposition}\label{Realization and Decomposition}
In this section, we show that every function in  $\mathcal {SA}_\infty(\mathcal E,\mathcal E_*)$ admits an Agler decomposition and a realization formula. For a nonempty set $X$, a function $K:X\times X\longrightarrow \mathcal L(\mathcal E)$ is called an \emph{$\mathcal L(\mathcal E)$-valued kernel on $X$} if for every finite choice of points
$x_1,\ldots,x_n\in X$ and vectors $h_1,\ldots,h_n\in\mathcal E$,
$$
\sum_{i,j=1}^{n}
\left\langle K(x_i,x_j)h_j,h_i\right\rangle_{\mathcal E} \geq 0.
$$
If $X$ is an open subset of one of the spaces $\mathbb C^d,$ $c_{00},$ $c_0$  or $\ell^\infty$, then, unless stated otherwise, the kernel $K$ is also assumed to be holomorphic in the first variable and anti-holomorphic in the second variable. The following theorem is the main result of this section.
\begin{theorem}\label{Auto-Ext}
Let  $f: \text{Ball}(c_{00})\rightarrow \mathcal L(\mathcal E, \mathcal E_*)$ be a function. Then the following are equivalent.
\begin{itemize}
\item [(1)] $f\in \mathcal {SA}_\infty(\mathcal E,\mathcal E_*)$.
        
\item [(2)] $($Agler decomposition on $\text{Ball}(\ell^\infty)$$)$. $f$ admits a holomorphic extension to $\text{Ball}(\ell^\infty)$ $($also denoted by $f$$)$ for which there exists a sequence $(K_j)_{j=1}^\infty$ of $\mathcal L(\mathcal E_*)$-valued kernels on $\text{Ball}(\ell^\infty)$ such that 
%each $K_j$ is holomorphic in the first variable and anti-holomorhic in the second varaiable, and
$$
I_{\mathcal E_*} - f(z)f(w)^* =\sum_{j=1}^\infty (1-z_j \overline{w_j}) K_j(z, w), \quad z, w \in \text{Ball}(\ell^\infty),
$$
where the series above converges in  SOT.

\item[(3)] $f$ is holomorphic on $\text{Ball}(c_{00})$, and there exists a sequence $(\Gamma_j)_{j=1}^\infty$ of $\mathcal L(\mathcal E_*)$-valued kernels on $\text{Ball}(c_{00})$ such that 
%each $\Gamma_j$ is holomorphic in the first variable and anti-holomorhic in the second varaiable, and
$$
I_{\mathcal E_*} - f(z)f(w)^* = \sum_{j=1}^\infty (1-z_j \overline{w_j}) \Gamma_j(z, w), \quad z, w \in \text{Ball}(c_{00}),
$$
where the series above converges in SOT. 

\item[(4)] $($Realization formula$)$. $f$ possesses  an extension to $\text{Ball}(\ell^\infty)$ $($also denoted by $f$$)$ for which there exist a sequence $(\mathcal H_j)_{j=1}^\infty$ of Hilbert spaces and a contraction $V = \begin{pmatrix}
A & B\\
C& D
\end{pmatrix}: \mathcal E \bigoplus \oplus_{j=1}^\infty \mathcal H_j \rightarrow  \mathcal E_* \bigoplus \oplus_{j=1}^\infty \mathcal H_j $ such that
$$
f(z) = A+ B(I-\Delta(z)D)^{-1}\Delta(z)C, \quad z \in \text{Ball}(\ell^\infty),
$$ where $\Delta(z) = \sum_{j=1}^\infty z_j P_j$ $($SOT$)$, and $P_k$ denotes the orthogonal projection of $\oplus_{j=1}^\infty \mathcal H_j$ onto $\mathcal H_k$.
\end{itemize}
\end{theorem}
%\textcolor{red}{If possible, give exmaple of $f\in \mathcal{SA}_\infty(\mathcal E,\mathcal E_*)$ for which, we cannot have Alger decomposition series convergence in norm.}
\begin{remark}%\textcolor{red}{uniqueness}
Note that, by Proposition \ref{Hilbert criterion}, every $f\in \mathcal {SA}_\infty(\mathcal E,\mathcal E_*)$ admits a unique holomorphic extension to $\text{Ball}(c_0).$ Theorem \ref{Auto-Ext} shows that $f$ extends holomorphically to the larger domain $\text{Ball}(\ell^\infty).$ However, such an extension need not be unique. For example, by the Hahn--Banach extension theorem, the zero function on $\text{Ball}(c_{00})$  admits distinct holomorphic extensions to $\text{Ball}(\ell^\infty).$
\end{remark}

We next show by an example that the SOT convergence in conditions (2) and (3) of Theorem \ref{Auto-Ext} cannot, in general, be replaced by norm convergence.

\begin{example}\label{exam}
 Let $f : \mathrm{Ball}(c_{00}) \longrightarrow \mathcal L\left(\bigoplus_{j=1}^{\infty} \mathcal H_j\right)$ be defined by
\[
f(z) = \sum_{j=1}^{\infty} z_j P_j \, ,
\]
where $P_k\in \mathcal L(\bigoplus_{j=1}^{\infty} \mathcal H_j)$ is the orthogonal projection onto $\mathcal H_k$. For any $n \in \mathbb N,$ note that
\begin{align*}
I_{\oplus \mathcal H_j} - \Big(\sum_{j=1}^n z_j P_j\Big) \Big(\sum_{j=1}^n \overline{w}_j P_j\Big)
&= I_{\oplus \mathcal H_j} - \Big(\sum_{j=1}^n z_j \overline{w}_j P_j\Big) \\
&= \left(\sum_{i=n+1}^{\infty} P_i\right) + \sum_{j=1}^n (1 - z_j \overline{w}_j) P_j\\
&= (1 - z_1 \overline{w}_1)\Big(\frac{\left(\sum_{i=n+1}^{\infty} P_i\right)}{1-z_1 \overline{w_1}} + P_1\Big) + \sum_{j=2}^n (1 - z_j \overline{w}_j) P_j.
\end{align*}
Thus, by Theorem~\ref{Agler decomposition realization}, $f|_{\mathbb D^n}$ is in Schur-Agler class for every $n$, and hence we have $f \in \mathcal {SA}_\infty(\bigoplus_{j=1}^{\infty} \mathcal H_j, \bigoplus_{j=1}^{\infty} \mathcal H_j)$.
We now claim that there does not exist a sequence $(\Gamma_j)_{j=1}^\infty$ of kernels such that
\begin{equation}\label{Normsum}
I - f(z) f(w)^* \;=\; \sum_{j=1}^\infty (1 - z_j \overline{w}_j) P_j ~~\text{(SOT)}~ \;=\; \sum_{j=1}^\infty (1 - z_j \overline{w}_j)\, \Gamma_j(z, w) ~~\text{(Norm)}, \quad z, w \in \text{Ball}(c_{00}). 
\end{equation}
Suppose, to the contrary, that such a sequence $(\Gamma_j)_{j=1}^\infty$ exists. Then, letting $z = w = 0$ gives that
\begin{equation}\label{at0}
I = \sum_{j=1}^\infty P_j ~~\text{(SOT)}~ = \sum_{j =1}^\infty \Gamma_j(0,0)~~\text{(Norm)}~.
\end{equation}
Fix $k\in \N.$ Let $e_k$ denote the sequence whose $k^{\text{th}}$ entry is $1$ and whose other entries are $0$. Then, for $z\in \D,$ we get 
\begin{equation*}
\sum_{j\neq k}P_j + (1-|z|^2)P_k= \sum_{j \neq k} \Gamma_j(ze_k,z e_k) ~~\text{(Norm)}~+ (1 - |z|^2 )\, \Gamma_k(ze_k,z e_k)
\end{equation*}
This implies that 
\begin{align*}
\sum_{j \neq k} \Gamma_j(ze_k,z e_k) ~~\text{(Norm)}~+ (1 - |z|^2 )\, \Gamma_k(ze_k,z e_k)\preceq I_{\bigoplus_{j=1}^{\infty} \mathcal H_j},
\end{align*}
and hence, 
\begin{equation}\label{diag}
\sum_{j \neq k} \Gamma_j(ze_k,z e_k) ~~\text{(Norm)}\preceq I_{\bigoplus_{j=1}^{\infty} \mathcal H_j}.\\[2mm]
\end{equation}
Let $z,w \in \D$. Then, by Eq. (\ref{Normsum}), $\sum_{j \neq k} \Gamma_j(ze_k,w e_k) ~~\text{(Norm)}$ exists. Let $E_k:=\{ze_k:z\in \D\}.$
Clearly, $\sum_{j \neq k} \Gamma_j$ is a kernel on $E_k\times E_k.$ Now, using Eq. (\ref{diag}), we get $\|\sum_{j \neq k} \Gamma_j(ze_k,w e_k)\|\leq 1$ for all $z,w\in \D.$ Similarly, for all $l\in \N,$ the kernel $\sum\limits_{\substack{j=1 \\j\neq k}}^{l} \Gamma_j$ satisfies 
\begin{equation}\label{finitesum}\Big\|\sum_{\substack{j=1 \\j\neq k}}^{l} \Gamma_j(ze_k,w e_k)\Big\|\leq 1
\end{equation}
for all $z,w\in \D.$ For $l\in \N,$ define $G,G_l:\D^2\rightarrow \mathcal L(\bigoplus_{j=1}^{\infty} \mathcal H_j)$ by
$$G_l(z,w):=\sum_{j=1,j\neq k}^{l} \Gamma_j(ze_k,\overline{w}e_k),\ \ \  G(z,w):=\sum_{j\neq k}\Gamma_j(ze_k,\overline{w}e_k).$$
By Eq. (\ref{finitesum}), $G_l$ is a uniformly bounded sequence of  holomorphic functions. 
Therefore, by \cite[Theorem 6.1]{Thuan}, the sequence $(G_l)_{l\in \N}$ converges uniformly on each compact subset of $\D^2$ to the holomorphic function $G.$ Using this uniform convergence, \cite[Proposition 15.5]{DGMS}, and the power series expansion
$$
\Gamma_j(ze_k,we_k)=\sum_{m,n\geq0} C_{j,m,n}z^m\overline{w}^{\,n},
$$
Eq. \eqref{Normsum} gives, for $z,w\in\D$, 
\begin{align*}
I_{\bigoplus_{j=1}^{\infty} \mathcal H_j} - z \overline{w} P_k
&= \sum_{j \neq k} \Gamma_j(ze_k, w e_k) ~~\text{(Norm)}~+ (1 - z \overline{w})\, \Gamma_k(ze_k, w e_k) \\[2mm]
&= \sum_{j \neq k} \Big( \sum_{m,n \geq 0} C_{j,m,n}  z^m \overline{w}^n \Big)
+ (1- z \overline{w}) \sum_{m,n \geq 0} C_{k,m,n} z^m \overline{w}^n
\\[2mm]
&=\sum_{m,n \geq 0} (\sum_{j \neq k} C_{j,m,n}) z^m \overline{w}^n
+ (1- z \overline{w}) \sum_{m,n \geq 0} C_{k,m,n} z^m \overline{w}^n.
\end{align*}
By the uniqueness of the coefficients of $z \overline{w}$, it follows that
\[ \Gamma_k(0,0) = P_k + \sum_{j=1}^{\infty} C_{j,1,1}.
\]
% \[
% \Rightarrow \quad  - P_k = \sum_{j=1}^{\infty} C_{j,1,1}
% \]
% \[
% \Rightarrow \quad \Gamma_k(0,0) \geq P_k \qquad \forall k
% \]
Since $\Gamma_j$ is a kernel, we have $C_{j,1,1}\succeq 0$ for $j\in \N.$ Therefore, $\Gamma_k(0,0) \succeq P_k$ for all $k\in \N.$ Now, by Eq. (\ref{at0}),
we finally get that $\Gamma_j(0,0) = P_j$ for all $j \geq 1$.
This is a contradiction.
\end{example}

Inspired by the ideas presented in \cite{Knese}, 
we establish a couple of lemmas and provide a proof of Theorem \ref{Auto-Ext}. Here, we introduce some useful notation. 
For a fixed sequence $\rho = (\rho_j)_{j = 1}^\infty \in \mathbb D_2^\infty$ of positive numbers, define $\mathbb D_\rho = \{z \odot \rho : z \in Ball(c_{0})\}$, where $z \odot \rho = (z_j\rho_j)_{j=1}^\infty$ for $z = (z_j)_{j=1}^\infty.$ Let $S, S_n:\D_2^{\infty}\times \D_2^{\infty}\to \mathcal L(\mathcal E_*)~(n \geq 1)$ be defined by 
\begin{equation}
\label{S,Sn}
    S(z,w)=\prod_{j=1}^{\infty}\frac{1}{1-z_j\overline{w}_j} I_{\mathcal E_*}, ~~ S_n(z,w)=\prod_{j\neq n}\frac{1}{1-z_j\overline{w}_j} I_{\mathcal E_*},~ z=(z_j)_{j=1}^\infty, w =(w_j)_{j=1}^\infty \in \mathbb D_2^\infty.
\end{equation}
Let $I$ denote the  constant function $I_{\mathcal E_*}$ on $\D_2^{\infty}\times \D_2^{\infty}.$
Again, we recall that, by Proposition~\ref{Hilbert criterion}, every
$f\in \mathcal{SA}_\infty(\mathcal E,\mathcal E_*)$ admits a unique holomorphic extension to $\text{Ball}(c_0)$.  
\begin{lemma}
\label{Realization-formula}
Let $f \in \mathcal {SA}_\infty(\mathcal E, \mathcal E_*)$. Then, there exists a sequence of $\mathcal L(\mathcal E_*)$-valued kernels $(K_j)_{j=1}^\infty$ on $\mathbb D_\rho$ such that, for each $x, y \in \mathcal E_*$ and $z,w \in \D_\rho,$
\begin{align*}
\sum_{j=1}^\infty |\langle{K_j(z, w)x, y}\rangle| < \infty,
\end{align*} and
$$
I_{\mathcal E_*} - f(z)f(w)^* =\sum_{j=1}^\infty (1-z_j \overline{w_j}) K_j(z, w), \quad z, w \in \mathbb D_\rho,  
$$
where the above series converges in the weak operator topology (WOT).
\end{lemma}
\begin{proof}
For each $N \geq 1$, define $f_N:\text{Ball}(c_0)\to \mathcal L(\mathcal E,\mathcal E_*)$  by 
$$f_N(z_1,z_2,\ldots)=f(z_1,z_2,\dots,z_N,0,0,\ldots).$$ Since $f|_{\D^N} \in \mathcal {SA}_n(\mathcal E,\mathcal E_*)$, by Theorem~\ref{Agler decomposition realization}, 
\begin{equation}
\label{decomposition-polydisc}
    I_{\mathcal E_*}- f_N(z)f_N(w)^*=\sum_{j=1}^{N} (1-z_j\overline{w_j}) K_j^N(z,w), \quad z = (z_j)_{j = 1}^\infty, w = (w_j)_{j=  1}^\infty \in \text{Ball}(c_0)
\end{equation}
for some $\mathcal L(\mathcal E_*)$-valued kernels $K_j^N~(j=1,2,\ldots,N)$ that depend only on the first $N$ variables.
%$z_1,\ldots,z_N,w_1,\ldots,w_N,$ are holomorphic in $z,$ and anti-holomorphic in $w$. 

Let $S, S_n$ be kernels defined as in \eqref{S,Sn}. Then, for all $n \geq \N$, $S,S_n\succeq I$ on $\D_2^{\infty}.$ Define $f_N\otimes f_N^*(z,w)=f_N(z)f_N(w)^*$ for $z,w\in \D_2^{\infty}$. Then, for any $M<N$, 
\begin{equation}
   \label{S-relation-partial sum} 
S\succeq (I-f_N\otimes f_N^*)S \overset{\eqref{decomposition-polydisc}}=\sum_{j=1}^{N}S_jK_j^N\succeq \sum_{j=1}^{N} K_j^N\succeq \sum_{j=1}^{M} K_j^N \quad \text{on $\D_2^\infty$}.
\end{equation} Therefore, by the Cauchy-Schwarz inequality, for any $x,y \in \mathcal E_*$,
\begin{align*}
|\langle K_j^N(z,w)x,y\rangle|^2=|\langle K_j^N(\cdot,w)x, K_j^N(\cdot,z)y\rangle|^2
&\leq \|K_j^N(\cdot,w)x\|^2 \|K_j^N(\cdot,z)y\|^2\\
&\leq \langle S(w,w)x,x\rangle \langle S(z,z)y,y\rangle\\
&=\|S(w,w)^\frac{1}{2}x\|^2 \|S(z,z)^\frac{1}{2}y\|^2, \quad z,w\in \D_2^{\infty},
\end{align*}
and hence, $\|K_j^N(z,w)\|\leq \|S(z,z)^\frac{1}{2}\| \|S(w,w)^\frac{1}{2}\|, ~z,w \in \D_2^{\infty}$,  $N\in \N,\  j\leq N$. Fix $\rho = (\rho_j)_{j = 1}^\infty \in \mathbb D_2^\infty$. Since $S(\rho \odot z, \rho \odot z)\preceq S(\rho, \rho)$, $z\in \text{Ball} (c_0),$ it follows for any $N\in \N$ that  $$\|K_j^N(\rho \odot z, \rho \odot \overline{w})\|\leq \|S(\rho, \rho)^\frac{1}{2}\|^2, \quad z,w \in \text{Ball} (c_0), 1\leq j\leq N. 
$$
Since $K_j^N$ depends only on the first $N$ variables, it follows from \cite[Remark 15.11]{DGMS} that the function $$\text{Ball}(c_0\times c_0)\ni (z,w)\rightarrow K_j^N(\rho \odot z, \rho \odot \overline{w})$$ is a holomorphic function.
Therefore, by Theorem \ref{Montel}, for each $j\in \N$, there exist a subsequence $N_{j,k}$ of $\mathbb N$ and a bounded function $K_j$ on $\D_\rho \times \D_\rho$ such that for all $x,y \in \mathcal E_*,$ $$\langle K_j^{N_{j,k}}(\rho \odot z, \rho \odot \overline{w})x,y\rangle \overset{k \rightarrow \infty}\longrightarrow \langle K_j (\rho \odot z, \rho \odot \overline{w})x,y\rangle$$ uniformly on each compact subset of $\text{Ball}(c_0\times c_0)$. Consequently, $K_j$ is a kernel. Now, by using a standard diagonal argument, we can find a common subsequence of $\mathbb N$, say the subsequence is $\mathbb N$ (without loss of generality), such that  for each $j \in \mathbb N$ and $x,y \in \mathcal E_*,$ \begin{align}\label{weaklyuniformly}\langle K_j^{N}(\rho \odot z, \rho \odot \overline{w})x,y\rangle \overset{N \rightarrow \infty}{\longrightarrow} \langle K_j (\rho \odot z, \rho \odot \overline{w})x,y\rangle
\end{align}
uniformly on each compact subsets of $\text{Ball}(c_0\times c_0)$. Thus, by letting $N\to \infty$ in \eqref{S-relation-partial sum}, we obtain that $S\succeq \sum_{j=1}^{M} K_j$ on $\D_\rho$ for all $M\in \N.$  In particular, for any $z\in \D_\rho,$ $S(z,z)\succeq \sum_{j=1}^{M} K_j(z,z),$ and since $(\sum_{j=1}^{M} K_j(z,z))_{M= 1}^\infty$ is an increasing sequence of positive operators bounded above by $S(z,z),$ it follows that $\sum_{j=1}^{\infty} K_j(z,z)$ exists in SOT. Hence, for any $x\in \mathcal E_*,$
\begin{align*}
\sum_{j=1}^{\infty} \|K_j(\cdot,z)x\|^2 =\sum_{j=1}^{\infty} \langle K_j(z,z)x,x\rangle
\leq \langle S(z,z)x,x\rangle
=\|{S(z,z)}^{\frac{1}{2}}x\|^2, \quad w\in \D_\rho.
\end{align*}
Therefore, for any $z,w\in \D_\rho$ and any $x,y \in \mathcal E_*,$ by the Cauchy-Schwarz inequality,  
\begin{align} \label{bound}
\nonumber 
\sum_{j=1}^{\infty} |\langle K_j(z,w)x,y\rangle|&\leq \sum_{j=1}^{\infty} \|K_j(\cdot,w)x\| \|K_j(\cdot,z)y\|\\ 
\nonumber &\leq \Big(\sum_{j=1}^{\infty} \|K_j(\cdot,w)x\|^2\Big)^{\frac{1}{2}}
\Big(\sum_{j=1}^{\infty} \|K_j(\cdot,z)y\|^2\Big)^{\frac{1}{2}}\\ 
&\leq \|{S(w,w)}^{\frac{1}{2}}x\| \|{S(z,z)}^{\frac{1}{2}}y\|. 
\end{align}
For $z,w \in \D_\rho,$ define the sesquilinear form $u(z,w): \mathcal E_*\times \mathcal E_* \to \C$ by $$u(z,w)(x,y):=\sum_{j=1}^{\infty}(1-z_j\overline{w}_j)\langle K_j(z,w)x,y\rangle, \quad x, y \in \mathcal E_*.$$ 
By Eq. \eqref{bound}, this form is well-defined and bounded. Hence, $\sum_{j=1}^{\infty} (1-z_j\overline{w}_j)K_j(z,w)$ exists in WOT for any $z,w \in \D_\rho$. Similarly, $\sum_{j=1}^{\infty} K_j(z,w)$ exists in WOT for any $z,w \in \D_\rho$.

By Proposition \ref{Hilbert criterion}, the function $f$ has a unique holomorphic extension $g$ to $\text{Ball}(c_0).$ Without abuse of notation, we call $g|_{\D_\rho}$ as $f.$ For $z,w \in \D_\rho,$ define $$K_0(z,w):=I_{\mathcal E_*}-f(z)f(w)^{*}- \sum_{j=1}^{\infty} (1-z_j\overline{w}_j) K_j(z,w).$$
It remains to verify that $K_0$ is a kernel. Let $Z_j\otimes Z_j^*$ be defined by $Z_j\otimes Z_j^*(z,w)=z_j\overline{w_j}$ for $z,w \in \D_2^{\infty}.$ First, note that, since $S\succeq K_j^N$ on $\D_2^\infty,$ it follows from Eq. (\ref{decomposition-polydisc}) that, for $M<N,$ 
\begin{align}\label{ineq2}
I -f_N\otimes {f^*_N}-\sum _{j=1}^{M}(I-Z_j\otimes Z_j^*)K_j^N+\sum_ {j=M+1}^{N}(Z_j\otimes  Z_j^*)S\succeq \sum_{j=M+1}^{N} K_j^N\succeq 0\quad \text{on $\D_2^\infty$}.   
\end{align}
Let  $z^{(i)} = (z_j^{(i)})_{j=1}^\infty$, ($1 \leq i \leq k)$ be vectors in $\D_\rho$ and $h_1,\ldots,h_k$ be in $\mathcal E_*$. As $\|S(\rho \odot z, \rho \odot w\|\leq \|S(\rho,\rho)\|,~z,w\in \text{Ball}(c_0)$ and $\sum_{r=M+1}^{N} Z_r\otimes Z_r^*$ is a kernel on $\D_\rho$, a direct calculation yields    
\begin{align}\label{ineq4}
0\leq \sum_{i,j=1}^{k} \sum_{r=M+1}^{N}z_r^{(i)}\overline{z_r^{(j)}}\langle S(z^{(i)},z^{(j)})h_j,h_i\rangle \leq \|S(\rho,\rho)\|\Big(\sum_{i=1}^{k}\|h_i\|\Big)^2\Big(\sum_{r=M+1}^{\infty}\rho_r^2\Big).
\end{align}
Then, by using Eqs. \eqref{ineq2} and (\ref{ineq4}),
%$I_{\mathcal E_*}-f_N\otimes {f^*_N}-\sum _{j=1}^{M}(I-Z_j\otimes Z_j^*)K_j^N+\sum_ {j=M+1}^{N}(Z_j\otimes  Z_j^*) \succeq 0S$ on $\D_\rho,$ by ,
\begin{align*}
\sum_{i,j=1}^{k}\Big(\langle (I_{\mathcal E_*}-f_N(z^{(i)}){f_N(z^{(j)})}^*)h_j,h_i\rangle& -\sum_{s=1}^{M}(1-z_s^{(i)}\overline{z_s^{(j)}})\langle K_s^N(z^{(i)},z^{(j)})h_j, h_i \rangle \Big)\\
&\geq -\|S(\rho,\rho)\|\Big(\sum_{i=1}^{k}\|h_i\|\Big)^2\Big(\sum_{r=M+1}^{\infty}\rho_r^2\Big).\end{align*}
Hence, by using Eq. (\ref{weaklyuniformly}) and the fact that for any $z\in \D_\rho,$ $f_N(z)$ converges to $f(z)$ in norm, letting $N\to \infty$ above yields that  
\begin{align*}
\sum_{i,j=1}^{k}\Bigg(\langle (I_{\mathcal E_*}-f(z^{(i)}){f(z^{(j)})}^*)h_j,h_i\rangle& -\sum_{s=1}^{M}(1-z_s^{(i)}\overline{z_s^{(j)}})\langle K_s(z^{(i)},z^{(j)})h_j, h_i \rangle \Bigg)\\
&\geq -\|S(\rho,\rho)\|\Big(\sum_{s=1}^{k}\|h_s\|\Big)^2\Big(\sum_{r=M+1}^{\infty}\rho_r^2\Big).
\end{align*}
Finally, by taking $M\to \infty,$ 
\begin{align*}
\sum_{i,j=1}^{k}\Bigg(\langle (I_{\mathcal E_*}-f(z^{(i)}){f(z^{(j)})}^*)h_j,h_i\rangle& -\sum_{s=1}^{\infty}(1-z_s^{(i)}\overline{z_s^{(j)}})\langle K_s(z^{(i)},z^{(j)})h_j, h_i \rangle \Bigg)\geq 0,
\end{align*}
proving that $K_0 \succeq 0$.
Absorbing the term $K_0$ into the \(j=1\) summand, we obtain

$$ I_{\mathcal E_*}-f(z)f(w)^* = (1-z_1\overline{w}_1) \left( \frac{K_0(z,w)}{1-z_1\overline{w}_1}+K_1(z,w) \right) +\sum_{j=2}^{\infty} (1-z_j\overline{w}_j)K_j(z,w). $$

Thus, upon renaming the kernels, we obtain the desired decomposition.
\end{proof}

\begin{lemma}
\label{Decom-to-realization}
    Let $X \subseteq \text{Ball}(\ell^\infty)$ and let $f: X \rightarrow \mathcal L(\mathcal E,\mathcal E_*)$ be a function. Suppose there exists a sequence of $\mathcal L(\mathcal E_*)$-valued kernels $(K_j)_{j=1}^\infty$ on $X$ such that for each $x, y \in \mathcal E_*$ and $z,w \in X,$
\begin{align}\label{abso}
\sum_{j=1}^\infty |\langle{K_j(z, w)x, y}\rangle| < \infty,
\end{align}
and
\begin{align}\label{aglerdec}
I_{\mathcal E_*} - f(z)f(w)^* = \sum_{j=1}^\infty (1-z_j \overline{w_j}) K_j(z, w), \quad z, w \in X,  
\end{align}
where the above series converges in WOT. Then there exist a sequence $(\mathcal H_j)_{j=1}^\infty$ of Hilbert spaces and a contraction $V = \begin{pmatrix}
    A & B\\
    C& D
\end{pmatrix}: \mathcal E \bigoplus \oplus_{j=1}^\infty \mathcal H_j \rightarrow  \mathcal E_* \bigoplus \oplus_{j=1}^\infty \mathcal H_j $ such that
$$
f(z) = A+ B(I-\Delta(z)D)^{-1}\Delta(z)C, \quad z \in X
$$
and for all $j \geq 1$,
$$
K_j(z, w) = B(I-\Delta(z)D)^{-1}P_j(I-D^*\Delta(w)^*)^{-1}B^*, \quad z, w \in X,
$$
where $P_k$ denotes the projection of $\oplus_{j=1}^\infty \mathcal H_j$ onto $\mathcal H_k$ and $\Delta(z) = \sum_{j=1}^\infty z_j P_j$ $($SOT$)$. Moreover, these  formulas extend $f$ and $K_j$ to $\text{Ball}(\ell^\infty)$ and there exists an $\mathcal L(\mathcal E_*)$-valued kernel $K_0$  on $\text{Ball}(\ell^\infty)$ such that 
$$
I_{\mathcal E_*} - f(z)f(w)^* = K_0(z, w) + \sum_{j=1}^\infty (1-z_j \overline{w_j}) K_j(z, w), \quad z, w \in \text{Ball}(\ell^\infty),
$$
where the series above converges in SOT.
\end{lemma}

Even if $X$ is an open subset, we do not assume that $K_j$ is holomorphic in the first variable or anti-holomorphic in the second variable in the hypotheses of Lemma \ref{Decom-to-realization}.

\begin{proof}[Proof of Lemma ~\ref{Decom-to-realization}]
Let $\mathcal H_j$ be the reproducing kernel Hilbert space with the kernel $K_j$ for all $j\in \N.$ For $z\in X$ and $y\in \mathcal E_*,$ let 
\begin{align*}
% \label{DefnV}
T(y, \overline{z_1}(K_1(\cdot,z)y), \overline{z_2}(K_2(\cdot,z)y),\ldots)=(f(z)^*y, K_1(\cdot,z)y, K_2(\cdot,z)y, \ldots)
\end{align*}
that extends linearly onto the $\text{span}\{(y, \overline{z_1}(K_1(\cdot,z)y), \overline{z_2}(K_2(\cdot,z)y),\ldots): z\in X, y \in \mathcal E_*\}$. By Eqs. (\ref{abso}) and  (\ref{aglerdec}), $T$ is a well-defined isometry, implying that $T$ extends as a contraction from $\mathcal E_* \bigoplus\oplus_{j=1}^{\infty}\mathcal H_j$ to $\mathcal E \bigoplus \oplus_{j=1}^{\infty}\mathcal H_j.$ Let 
\begin{equation}
\label{T-blockmatrix}
T=\begin{pNiceMatrix}[first-row,last-col,nullify-dots]
\mathcal E_* & \oplus_{j=1}^{\infty} \mathcal H_j & \\
A^* &C^* & \mathcal E\\ B^*& D^*& \oplus_{j=1}^{\infty} \mathcal H_j.
\end{pNiceMatrix}
\end{equation}
be in block matrix form. For $z\in X \subseteq \text{Ball}(\ell^\infty),$ define $F(z)^*: \mathcal E_*\to \oplus_{j=1}^{\infty}\mathcal H_j$ by $$F(z)^*(y)=(K_j(\cdot,z)y)_{j=1}^\infty.$$ Since $z\in X,$ by Eq. (\ref{aglerdec}), we get that $F(z)^*$ is a bounded linear operator. For $z\in \text{Ball}(\ell^\infty),$ let $$\Delta (z)((h_j)_{j=1}^\infty):=(z_jh_j)_{j= 1}^\infty,~ (h_j)_{j =1}^\infty \in \oplus_{j=1}^{\infty}\mathcal H_j.$$ Then,
\begin{align}\label{shortDefnV}
T\begin{pmatrix}
    I_{\mathcal E_*} \\
    \Delta(z)^*F(z)^*
\end{pmatrix}(y)= \begin{pmatrix}
    A^* & C^*\\
    B^*& D^*
\end{pmatrix}\begin{pmatrix}
    y\\
    \Delta(z)^*F(z)^*y
\end{pmatrix} =(f(z)^*y, F(z)^*y)
\end{align} for any $z\in X$ and $y\in \mathcal E_*.$ This implies that
\begin{align}
    \label{two-equalities}
    f(z) = A + F(z)\Delta(z)C\quad \& \quad  F(z) = B+ F(z)\Delta(z)D \quad \text{for all $z \in X$}.
\end{align}
Consequently,
\begin{align}\label{formula F}F(z)=B(I-\Delta(z)D)^{-1}\end{align}
and \begin{align}\label{formula f}f(z)=A+B(I-\Delta(z)D)^{-1}\Delta(z)C
\end{align}for any $z\in X.$ Since the right-hand sides of Eqs. (\ref{formula F}) and (\ref{formula f}) are well-defined for $z \in \text{Ball}(\ell^\infty)$, these same formulas naturally extend the functions $F$ and $f$ to the entire domain $z \in \text{Ball}(\ell^\infty)$.
% Observe that the right hand sides of the eqs. (\ref{formula F}) and (\ref{formula f}) are defined for $z \in \text{Ball}(\ell^\infty),$ which implies that the functions $f$ and $F$ and $f$ as given in  eqs. (\ref{formula F}) and (\ref{formula f}) extend to  $z \in \text{Ball}(\ell^\infty)$. 
% It is easy to see that the eq. (\ref{shortDefnV}) is valid for any $z\in \text{Ball}(\ell^\infty)$ and $y\in \mathcal E_*.$ 
Moreover, since for any $x,y\in \mathcal E_*$,
\begin{align*}
    \langle F(z)(0, \ldots,0, K_j(\cdot,w)y,0, \ldots),x \rangle &=\langle (0,\ldots,0,K_j(\cdot,w)y,0,\ldots), F(z)^*x\rangle\\
    &= \langle K_j(z,w)y,x\rangle, \quad j\in \N, z,w \in X,
    \end{align*}
    it follows that $F(z)(0, \ldots,0, K_j(\cdot,w)y,0, \ldots)=K_j(z,w)y$, and hence $$F(z)P_jF(w)^*=K_j(z,w), \quad z,w \in X, j\in \N,$$ where $P_j$ is the orthogonal projection of $\oplus_{k=1}^{\infty} \mathcal H_k$ onto $\mathcal H_j.$ Since $F$ is defined on $\text{Ball}(\ell^\infty)$, it then follows that $F(z)P_jF(w)^*$ is a kernel on $\text{Ball}(\ell^\infty)$, and hence, $K_j$ also extends to $\text{Ball}(\ell^\infty)$ for all $j\in \N.$   

Let $W$ be a contraction such that $I-T^*T=W^*W.$ Then, for $y \in \mathcal E_*$,
$$
\|W(y, \Delta(z)^*F(z)^*y)\|^2\leq \|y\|^2 + \|\Delta(z)^*F(z)^*y\|^2, \quad z \in \text{Ball}(\ell^\infty).
$$
This allows us to define the function $G:\text{Ball}(\ell^\infty)\to \mathcal L(\mathcal E_*, \mathcal E_*\bigoplus \oplus_{j=1}^{\infty}\mathcal H_j)$ given by $$G(z):=W \begin{pmatrix}
    I_{\mathcal E_*}\\
    \Delta(z)^*F(z)^*
\end{pmatrix}, \quad z \in \text{Ball}(\ell^\infty).$$ 
% It is easy to see that $G(z)$ is a bounded operator for all $z\in \text{Ball}(\ell^\infty),$ 
Note that $G$ is an anti-holomorphic function, and since the extended functions $f, F$ satisfy \eqref{two-equalities} on $\text{Ball}(\ell^\infty)$, we have for any $z, w \in \text{Ball}(\ell^\infty)$ that
\begin{align*}
G(z)^*G(w)
&= \begin{pmatrix}
    I_{\mathcal E_*}
 &     F(z)\Delta(z)
\end{pmatrix} (I-T^*T) \begin{pmatrix}
    I_{\mathcal E_*} \\
      \Delta(w)^*F(w)^*
\end{pmatrix} \notag\\
&\stackrel{\eqref{shortDefnV}}{=} I_{\mathcal E_*}+ F(z)\Delta(z)\Delta(w)^*F(w)^* -   \begin{pmatrix}
    I_{\mathcal E_*}
 &     F(z)\Delta(z)
\end{pmatrix} T^*\begin{pmatrix}
    f(w)^* \\
    F(w)^*
\end{pmatrix}\notag\\
&\stackrel{\eqref{T-blockmatrix}}{=} I_{\mathcal E_*}+ F(z)\Delta(z)\Delta(w)^*F(w)^* -   \begin{pmatrix}
    A+F(z)\Delta(z)C &  B+  F(z)\Delta(z)D
\end{pmatrix} \begin{pmatrix}
    f(w)^* \\
    F(w)^*
\end{pmatrix}\\
&= I_{\mathcal E_*}+ F(z)\Delta(z)\Delta(w)^*F(w)^* -
    f(z)f(w)^* -  F(z)F(w)^*.
\end{align*}
% Now, $T$ as the block matrix and using eqs. (\ref{formula F}) and (\ref{formula f}), 
Thus, we obtain that $$I_{\mathcal E_*}-f(z)f(w)^*=F(z)(I-\Delta(z)\Delta(w)^*)F(w)^*+G(z)^*G(w), \quad z, w\in \text{Ball}(\ell^\infty).$$ 
Furthermore, note that $$F(z)(I-\Delta(z)\Delta(w)^*)F(w)^*=\sum_{j=1}^{\infty}(1-z_j\overline{w_j})F(z)P_jF(w)^*,$$ where the above series converges in SOT. Hence, the proof now concludes by taking $K_0(z,w)=G(z)^*G(w)$ for $z,w\in \text{Ball} (\ell^\infty)$ and $V=T^*$.
\end{proof}

\begin{proof}[Proof of Theorem~\ref{Auto-Ext}]
$(1) \implies (2)$: First, observe that if two holomorphic functions on $\text{Ball}(c_0)$ agree on $\D_\rho$, then they coincide on  $\text{Ball}(c_0)$. The desired implication now follows from Proposition \ref{Hilbert criterion}, Lemma \ref{Realization-formula}, and  Lemma \ref{Decom-to-realization} with $X=\D_\rho$.

$(2) \implies (3)$: It is obvious.

$(3) \implies (1)$: This follows from Theorem \ref{Agler decomposition realization}.

$(1)\iff (4)$: The forward implication is an immediate consequence of Lemmas \ref{Realization-formula} and \ref{Decom-to-realization}. The reverse implication follows from Theorem \ref{Agler decomposition realization}.
\end{proof}
Next, we establish a key ingredient for an Agler-Pick interpolation theorem in infinitely many variables. In \cite{AM}, the authors studied admissible kernels on the bidisc $\D^2$ and characterized the existence of a certain representation for an $\mathcal L(\mathcal E_*)$-valued self-adjoint function on $\D^2\times \D^2$  using $\mathcal L(\mathcal E_*)$-valued admissible kernels; see \cite[Theorem 11.26]{AM}. We generalize the notion of admissible kernels to subsets of $\D_2^\infty$ and characterize when an $\mathcal L(\mathcal E_*)$-valued self-adjoint function, defined on a finite subset of $\D_2^\infty$, admits the representation (\ref{RepresntationG}); see Theorem~\ref{AdmissibleDecomposition}. The proof uses ideas from \cite{AM} and \cite{Knese}. 

\begin{definition}
Given a set $X \subseteq \D_2^\infty$, let $K: X \times X \rar \mathcal L(\mathcal E_*)$ be a kernel. The kernel $K$ is called admissible if $(I-Z_j\otimes Z_j^*)K \succeq 0$ for all $j\in \N.$    
\end{definition}

Let $X$ be a set. For functions $G, K: X  \to \mathcal L(\mathcal E),$ let $G\oslash K: X \to \mathcal L(\mathcal E\otimes \mathcal E)$ be defined by
$$G\oslash K(z):=G(z)\otimes K(z), \quad z \in X.$$ For self-adjoint matrices $A$ and $B$, the notation $A \succeq B$ means that $A-B$ is positive semidefinite.

\begin{theorem}\label{AdmissibleDecomposition}
Let $X\subset \D_2^\infty$ be a finite subset, and let $G:X\times X\to \mathcal L(\mathcal E_*)$ be a  self-adjoint function, i.e., $G(z,w)=G(w,z)^*,~z,w\in X.$ Then, $G\oslash K$ is an $\mathcal L(\mathcal E_* \otimes \mathcal E_*)$-valued kernel for every admissible $\mathcal L(\mathcal E_*)$-valued kernel $K$ on $X$ if and only if there exists a sequence of $\mathcal L(\mathcal E_*)$-valued kernels $(K_j)_{j=1}^\infty$ on $X$ such that 
\begin{align}\label{RepresntationG}
G(z,w) = \sum_{j=1}^\infty (1-z_j \overline{w_j}) K_j(z, w), \quad z, w \in X,
\end{align}
where the series above converges in SOT.
\end{theorem}
\begin{proof}
The sufficiency is straightforward. For the converse part, let $X=\{x_1,\ldots,x_N\}$ with cardinality $N$. Then, we may view any $\mathcal L(\mathcal E_*)$-valued function $g$ on $X \times X$ as the matrix $(g(x_i,x_j))_{i,j=1}^N \in \mathcal L(\oplus_{i=1}^{N} \mathcal E_*).$ Consider the cone 
\begin{equation*}
\mathscr C = \Big\{(z,w) \to \sum_{j=1}^\infty (1-z_j \overline{w_j}) A_j(z,w)~\text{(SOT)} : A_j: X \times X \to \mathcal L(\mathcal E_*)~\text{is a kernel $\forall~ j \geq 1$} \Big\}.
\end{equation*}
Note that $\mathscr C$ is convex. We first claim that 
\begin{equation}
\label{weak-star}
  \text{ $\mathscr C$ is closed in the weak star topology.}  
\end{equation}
By the Krein–Šmulian theorem and the Banach–Alaoglu theorem, it suffices to show that, for every $d>0$, the set $\mathscr C_d:=\mathscr C \cap \{A\in \mathcal L(\oplus_{i=1}^{N} \mathcal E_*): \|A\|\leq d\}$ is weak star closed in the closed ball $\{A\in \mathcal L(\oplus_{i=1}^{N} \mathcal E_*): \|A\|\leq d\}$. Since $\mathcal E_*$ is separable, the weak star topology on each such ball is metrizable. In view of this, let $\{C_n\}_{n=1}^\infty \subseteq \mathscr C_d$ be a sequence which converges to $C$ in weak star topology, and let $C_n$ be of the form
 \begin{equation}
 \label{Cn-formula}
     C_n(z,w) = \sum_{j=1}^\infty (1-z_j \overline{w_j}) A_{n,j}(z,w)~\text{(SOT)}, \quad z, w \in X,
 \end{equation}
where $A_{n,j}$ are $\mathcal L(\mathcal E_*)$-valued kernels for all $j, n \in \N$.  
 Note that for all $n \geq 1$, $C_n$ is a self-adjoint matrix; that is, $C_n(w,z) = C_n(z,w)^*,~z,w \in X$, and since $\{C_n\}_{n=1}^\infty$ is bounded, it follows that $$d I \succeq C_n \succeq -dI \quad \text{for all $n \geq 1$}.$$ 
% Note that $S,S_n$ are kernels.
Let $\delta: X \times X \to \mathcal L(\mathcal E_*)$ be defined by
 \begin{equation*}
     \delta(z, w) 
     = \begin{cases}
         I_{\mathcal E_*}, & \text{if}~z=w,\\
         0, & \text{otherwise}.
     \end{cases}
 \end{equation*} 
 Let $S, S_n$ be kernels given in \eqref{S,Sn}. If $d': = \max\limits_{1\leq j\leq N} S(x_j,x_j)$, then by using $S,S_n \succeq  I$ and \eqref{Cn-formula},  
\begin{equation}
\label{relation-positivity}
d' \delta \succeq d\delta S\succeq C_n S = \sum_{j=1}^\infty S_j A_{n,j} \succeq \sum_{j=1}^m S_j A_{n,j}\succeq \sum_{j=1}^m A_{n,j} \succeq A_{n,m}, \quad m, n \in \mathbb N.    
\end{equation}
 % by the finiteness of the set $X$.
 This shows that the sequence $\{A_{n,m}\}_{m,n=1}^\infty \subseteq \mathcal L(\oplus_{j=1}^N \mathcal E_*)$ of operators is uniformly bounded; that is, $\|A_{n,m}\| \leq d'$ for all  $n, m \in \N$. By a standard diagonal argument, there exists a common subsequence of $\mathbb N$, say the subsequence is $\mathbb N$ (without loss of generality), such that for each $m \in \mathbb N$, 
\begin{equation}
\label{Diagonal-seq}
   A_{n, m} \overset{n \to \infty}{\longrightarrow} A_m~\text{in WOT}.
\end{equation}
This implies that the limit $A_m$ is a positive semidefinite matrix. Furthermore, since $C_n \overset{w^*}{\longrightarrow} C$ as $n \to \infty$ gives that $C_n \overset{n \to \infty}{\longrightarrow} C$ in WOT, by  \eqref{relation-positivity} and \eqref{Diagonal-seq},
$$
CS \succeq \sum_{j=1}^m S_jA_j \succeq \sum_{j=1}^m A_j.
$$
Thus, the sequence $\{\sum_{j=1}^m S_jA_{j}\}_{m = 1}^\infty$ is increasing and bounded above by $CS$, implying that $\sum_{j=1}^\infty S_jA_j$ converges in SOT. Therefore, 
\begin{equation}\label{SOT-con-diff}
    \sum_{j=1}^\infty (I- Z_j \otimes Z_j^*)A_j~\text{converges in SOT}.
\end{equation}
Fix $n \in \mathbb N$. As above, it also follows from \eqref{relation-positivity} that the sequence $\{\sum_{j=1}^m A_{n, j}\}_{m = 1}^\infty$ is increasing and bounded above by $d'\delta$, which shows that $\sum_{j=1}^\infty A_{n, j}$ converges in SOT. Therefore, since $X \subseteq \ell^2$, the series $\sum_{j=1}^\infty (Z_j \otimes Z_j^*) \delta$ converges in SOT, and hence, for all $m \geq 1,$
\begin{eqnarray*}
&&\sum_{j=m+1}^\infty (I - Z_j \otimes Z_j^*)A_{n,j}~\text{(SOT)} + d' \sum_{j=m+1}^\infty (Z_j \otimes Z_j^*) \delta ~\text{(SOT)} \notag\\
     &&\overset{\eqref{relation-positivity}}\succeq \sum_{j=m+1}^\infty (I - Z_j \otimes Z_j^*)A_{n,j}~\text{(SOT)} + \sum_{j=m+1}^\infty (Z_j \otimes Z_j^*) A_{n,j} ~\text{(SOT)}\notag\\ 
     &&=\sum_{j=m+1}^\infty A_{n,j}~\text{(SOT)} \succeq 0 \quad \text{on $X$}.
\end{eqnarray*}
This together with \eqref{Cn-formula} yields that for $m \geq 1,$ 
\begin{eqnarray*}
    C_n - \sum_{j=1}^m (I - Z_j \otimes Z_j^*)A_{n,j} 
     \succeq  -d' \sum_{j=m+1}^\infty (Z_j \otimes Z_j^*) \delta ~\text{(SOT)}~  
     \succeq -d' \max\Big\{\sum_{j=m+1}^\infty |z_j|^2 : z \in X \Big\} \delta.
\end{eqnarray*}
% Since $C_n \overset{w^*}{\longrightarrow} C$ as $n \to \infty$, we get that $C_n \to C$ weakly as $n \to \infty$,
Now, by \eqref{Diagonal-seq}, letting $n \to \infty$ in the expression above yields  
\begin{equation} 
\label{limit-after-exp1}
    C - \sum_{j=1}^m (I - Z_j \otimes Z_j^*)A_{j}
     \succeq -d' \max\Big\{\sum_{j=m+1}^\infty |z_j|^2 : z \in X \Big\} \delta.
\end{equation}
 Since $X \subseteq \ell^2$ is finite, $\max\big\{\sum_{j=m+1}^\infty |z_j|^2 : z \in X \big\}$ converges to $0$ as $m \to \infty$, and hence, by \eqref{SOT-con-diff} and sending $m \to \infty$ in \eqref{limit-after-exp1},
$$
 C - \sum_{j=1}^\infty (I - Z_j \otimes Z_j^*)A_{j}~\text{(SOT)} 
     \succeq 0.
$$
This finally gives that $C \in \mathscr C_d$, proving the claim \eqref{weak-star}. 

The remaining part of the argument is standard. Let $G\oslash K$ be a kernel for every admissible $\mathcal L(\mathcal E_*)$-valued kernel $K$ on $X.$ Suppose $G\notin \mathscr C.$ Then, by the Hahn-Banach separation theorem, there exists a self-adjoint trace class operator $T=(T_{i,j})_{i,j=1}^ N \in \mathcal L(\oplus_{i=1}^{N}\mathcal E_*)$ such that \begin{equation}
   \label{contr-assump}
   tr(TG)<0\leq tr (TC) \quad \text{for all $C\in \mathscr C$}.
\end{equation}
For any $h =(h_j)_{j=1}^N \in \oplus_{j=1}^N \mathcal E_*$,  let $a:X\to \mathcal E_*$ be given by $a(x_j) = h_j~ (1\leq j \leq N)$ and let $T_a: X\times X\to \mathcal L(\mathcal E_*)$ be  denoted by \begin{align}\label{defnofTa}T_a(x_i,x_j)=a(x_i)\otimes a(x_j), \quad 1\leq i, j\leq N,
\end{align}
where the right-hand side denotes a rank-one operator. Then 
$$
\langle Th, h\rangle = \sum_{i,j=1}^{N} \langle T_{i,j}h_j, h_i\rangle = tr(TT_a).
$$
Since $T_a$ is a kernel and $T_a = (I-Z_1\otimes Z_1^*) \frac{T_a}{I-(Z_1\otimes Z_1^*)}$, we have $T_a\in\mathscr C.$ Therefore, by \eqref{contr-assump}, $\langle Th, h\rangle \geq 0$.
% let $a\otimes a^*: X\times X\to \mathcal L(\mathcal E_*)$ denote the function $$a\otimes a^*(x_i,x_j)=a(x_i)\otimes a(x_j).$$ 
% Let $C=(I-Z_1\otimes Z_1^*) \frac{a\oslash a^*}{I-(Z_1\otimes Z_1^*)}.$ Then $C\in \mathscr C$, and hence $$tr(TC)=tr(T)=\sum_{i=1}^{N}\sum_{j=1}^{N} \langle T_{i,j}a(x_j), a(x_i)\rangle.$$ 
Thus,  $T$ is a positive trace class operator.

Let $\{e_k:k\in \N\}$ be an orthonormal basis of $\mathcal E_*.$ Define $K^t: X\times X\to \mathcal L(\mathcal E_*)$ by $$K^t(x_i, x_j)=T_{j,i}^t, \quad 1\leq i,j\leq N,$$ where $T_{j,i}^t$ denotes the transpose of $T_{j,i}$ with respect to the orthonormal basis $\{e_k:k\in \N\}$. Let $u =\sum_{k=1}^{\infty}c_ke_k, v = \sum_{k=1}^{\infty}d_ke_k$ be in $\mathcal E_*$, and
define $\overline{u}=\sum_{k=1}^{\infty}\overline{c_k}e_k$ and $\overline{v}=\sum_{k=1}^{\infty}\overline{d_k}e_k.$ Since for any $k,l\in \N,$ $\langle T_{j,i}^t(e_k),e_l\rangle=\langle T_{j,i}(e_l),e_k\rangle,$ we get that 
\begin{align}\label{Transbar}\langle K^t(x_i, x_j)u,v\rangle=\langle T_{j,i}(\overline{v}), \overline{u}\rangle
\end{align}
for all $1\leq i,j\leq N.$ Combining this with the positivity of $T$ gives that $K^t$ is a kernel.

For each $k\in\mathbb N$ and each function $a:X\to\mathcal E_*$, define $C_{k,a}:=(I-Z_k\otimes Z_k^*) (T_a),$ where the kernel $T_a$ is defined as in Eq. (\ref{defnofTa}). Then, $C_{k,a} \in \mathscr C$. Hence, a direct computation shows that 
\begin{align*}0\leq tr(TC_{k,a})&=\sum_{i=1}^{N}\sum_{j=1}^{N} (1-x_j(k)\overline{x_i(k)})\langle T_{i,j}a(x_j), a(x_i)\rangle
\end{align*}
where $x_j = (x_j(k))_{k=1}^\infty$ for all $j =1, \ldots, N$.
This together with Eq. (\ref{Transbar}) implies that $K^t$ is an admissible kernel. So, by the assumption, $G\oslash K^t$ is a kernel. In particular,  
% Since $G\otimes K^t$ is  positive semidefinite, 
for any finite subset  $F$ of $\N,$ we get that
\begin{align}
\label{finitepositive}
\sum_{i,j=1}^{N}\sum_{k, l \in F}\langle G(x_i,x_j)e_k, e_l\rangle \langle K^t(x_i,x_j)e_k, e_l \rangle\geq 0.
\end{align}
It is easy to see that $$tr(TG)=\sum_{i,j=1}^{N} \sum_{l=1}^\infty \sum_{k=1}^\infty \langle G(x_i,x_j)e_k,e_l\rangle \langle K^t(x_i, x_j)e_k,e_l\rangle.$$
Since each $T_{i,j}$ is trace class and Eq. (\ref{finitepositive}) holds for every finite subset $F,$ it follows that $tr(TG)\geq 0$, which is a contradiction. Hence, $G$ admits a representation as in Eq. (\ref{RepresntationG}). This completes the proof.
\end{proof}

\section{Agler-Pick interpolation}\label{Agler-Pick}

% Below we present the main result of this section, which provides a version of Agler-Pick interpolation type theorem in operator-valued setting. This extends the scalar case \cite[Theorem~1.7]{Knese}.
In this section, we present a version of Agler-Pick interpolation type theorem in infinitely many variables in operator-valued setting. This extends the scalar case \cite[Theorem~1.7]{Knese}.
\begin{theorem}\label{A_P_infinite}
 Let $X$ be a finite subset of $\text{Ball}(\ell^\infty)$ and $f: X \to \overline{\text{Ball}}(\mathcal L(\mathcal E,\mathcal E_*))$ be a function. Consider the following conditions.
 \begin{itemize}
     \item [(i)] There exists $\tilde{f} \in \mathcal {SA}_\infty(\mathcal E, \mathcal E_*)$ with $\tilde{f}|_{X} = f$.
     \item [(ii)] There exists a sequence of $\mathcal L(\mathcal E_*)$-valued kernels $(K_j)_{j=1}^\infty$ on $X$ such that 
        $$
I_{\mathcal E_*} - f(z)f(w)^* = \sum_{j=1}^\infty (1-z_j \overline{w_j}) K_j(z, w), \quad z, w \in X,
$$
where the series above converges in SOT.

\item [(iii)] The function $(I_{\mathcal E_*}-f\otimes f^*)\otimes K$ is a kernel for every admissible $\mathcal L(\mathcal E_*)$-valued kernel $K$ on $X$, where $f\otimes f^*$ denotes the function $f\otimes f^*(z,w)=f(z)f(w)^*.$
% \item [(iv)] \textcolor{red}{For every tuple $T = (T_1, T_2, \ldots)$ of commuting, contractive, simaltaneously diagonalizable matrices whose joint eigenspaces have dimensnions at most $1$ and satisfy $\sigma(T) \subseteq X$, we have $\|f(T)\| \leq 1$.}
 \end{itemize}
 Then $(i)$ and $(ii)$ are equivalent and imply $(iii)$, while $(iii)$ implies $(ii)$ whenever $X \subseteq \mathbb D_2^\infty$.
\end{theorem}
\begin{proof}
$(i) \iff (ii)$: By Theorem~\ref{Auto-Ext}(2), the forward implication is immediate. The other implication follows from \textcolor{blue}{Lemma~\ref{Decom-to-realization}.}

$(ii) \implies (iii)$: This is trivial.

$(iii) \implies (ii)$: This is immediate by Theorem \ref{AdmissibleDecomposition} by taking $G=I_{\mathcal E_*}-f\otimes f^*.$
\end{proof}

\noindent\textbf{Acknowledgement:} The authors express their deepest gratitude to Prof. Bata Krishna Das for his invaluable guidance and numerous insightful and stimulating discussions throughout the development of this work. The first named author is supported by a National Postdoctoral Fellowship (N-PDF) provided by the ANRF, India
(File number: PDF/2025/006823), and an Institute postdoctoral fellowship from the Indian Institute of Technology Bombay. The second named author acknowledges the support received from the Indian Institute of Technology Bombay through an Institute postdoctoral fellowship. 

\vspace{0.1in}
\noindent \textbf{AI Declaration:}
AI has been used to gain an idea of why Example \ref{exam} works and to improve the writing of the manuscript.

\vspace{0.1in}
\noindent\textbf{Data Availability:}
Data sharing is not applicable to this article, as no datasets were generated or analyzed during
the current study. In case any datasets are generated during and/or analyzed during the current study, they must
be available from the corresponding author on reasonable request.

\vspace{0.1in}
\noindent\textbf{Conflict of interest:} There is no conflict of interest.

\end{document}